\documentclass[10pt]{amsart}
\usepackage{amssymb}
\usepackage{stmaryrd}

\newtheorem{thm}{Theorem}[section]
\newtheorem{lemma}[thm]{Lemma}
\newtheorem{proposition}[thm]{Proposition}

\newtheorem{definition}[thm]{Definition}

\newcommand{\p}{\mathbb{P}}
\newcommand{\q}{\mathbb{Q}}

\newcommand{\dom}{\mathrm{dom}}
\newcommand{\ran}{\mathrm{ran}}

\newcommand{\h}{\mathrm{ht}}

\newcommand{\res}{\upharpoonright}

\begin{document}

\title{Suslin tree preservation and club isomorphisms}

\author{John Krueger}

\address{John Krueger \\ Department of Mathematics \\ 
	University of North Texas \\
	1155 Union Circle \#311430 \\
	Denton, TX 76203}
\email{jkrueger@unt.edu}

\date{February 2022; revised December 2022}



\thanks{This material is based upon work supported by the Simons Foundation under Grant 631279}

\begin{abstract}
	We construct a model of set theory 
	in which there exists a Suslin tree and satisfies that any two 
	normal Aronszajn trees, neither of which contains a Suslin subtree, are club isomorphic. 
	We also show that if $S$ is a free normal Suslin tree, then for any positive integer $n$ 
	there is a c.c.c.\ forcing extension in which 
	$S$ is $n$-free but all of its derived trees of dimension greater 
	than $n$ are special.
\end{abstract}

\maketitle

\section{Introduction}

Baumgartner \cite{BDISS} proved that Martin's axiom implies that  
all Aronszajn trees are special. 
This consequence of Martin's axiom 
in turn implies Suslin's hypothesis, since any Suslin tree is 
a nonspecial Aronszajn tree. 
Later Abraham-Shelah \cite{AS} constructed a model 
in which there exists a Suslin tree and any Aronszajn tree 
which does not contain a Suslin subtree is special. 
Abraham-Shelah \cite{AS} also introduced the property that any two normal Aronszajn 
trees are club isomorphic, which implies that all Aronszajn trees are special, 
and proved its consistency from \textsf{ZFC} and that it follows 
from the proper forcing axiom.

In light of these results, a natural question is whether there is a 
variation of the property that all normal Aronszajn 
trees are club isomorphic which is consistent with the existence of a Suslin tree. 
In this article we answer this question by constructing a model in which 
there exists a Suslin tree and any two 
normal Aronszajn trees, neither of which contains a Suslin subtree, are club isomorphic. 
We also prove that this statement follows from Todorcevic's forcing axiom 
$\textsf{PFA}(S)$.

The main method which we use is that of preserving a Suslin tree $S$ after forcing something 
about an Aronszajn tree $T$ which is in some sense not near it. 
More specifically, we are concerned with the relation that 
forcing with $S$ below some element $x \in S$ 
adds an uncountable branch to the Aronszajn tree $T$, or equivalently, 
that there exists a strictly increasing and height preserving 
map from a club set of levels of $S_x$ into $T$. 
It turns out that if there does not exist such a map, then we can force things about 
$T$ while preserving $S$ being Suslin. 

In addition to our main result discussed above, we give another application 
of the idea of Suslin tree preservation to the topic of free trees. 
We show that if $S$ is a free Suslin tree, then for any positive integer $n$ there 
exists a c.c.c.\ forcing which forces that $S$ is $n$-free, but all of the 
derived trees of $S$ of dimension greater than $n$ are special. 
This shows that in contrast to the property of homogeneity of Suslin trees, which 
is upwards absolute, freeness is highly malleable by forcing.

\section{Preliminaries}

We assume that the reader is familiar with Aronszajn and Suslin trees, as well as 
the basics of forcing and forcing axioms. 
In this section we go over our notation, review some standard background results, 
and prove some elementary lemmas which we need later.

All of the trees we discuss in this article have height $\omega_1$. 
An \emph{$\omega_1$-tree} is a tree of height $\omega_1$ which has countable levels. 
We write $\h_T(x)$ for the height of an element $x$ of a tree $T$, 
$T_\alpha$ for level $\alpha$ of $T$ consisting of all $x$ with $\h_T(x) = \alpha$, 
$T \res \alpha = \bigcup \{ T_\beta : \beta < \alpha \}$, 
and more generally, $T \res C$ for the set 
of $x$ in $T$ with $\h_T(x) \in C$. 
If $\alpha < \h_T(x)$ we write $x \res \alpha$ for the unique $y <_T x$ with height $\alpha$. 
For incomparable elements $x$ and $y$ of $T$, 
$\Delta_T(x,y)$ is the order type of the set 
of $z$ below both $x$ and $y$.

A \emph{branch} of a tree $T$ is a maximal chain. 
If $B$ is a branch and $\alpha$ is an ordinal 
less than its order type, $B(\alpha)$ denotes the 
element of $B$ of height $\alpha$. 
An \emph{antichain} of $T$ is a set of pairwise incomparable elements of $T$. 
A tree $T$ is \emph{normal} if it has a root, every element of $T$ has at least two 
immediate successors, every element of $T$ has some element above it at any higher level, 
and there is at most one upper bound to any chain of $T$ whose order type is a limit ordinal. 
A \emph{subtree} of $T$ is any subset of $T$ with the order inherited from $T$.

An \emph{Aronszajn tree} is an $\omega_1$-tree with no cofinal branch, and a 
\emph{Suslin tree} is a tree with no uncountable chain or antichain.  
A tree $T$ of height $\omega_1$ is \emph{special} if it has a 
\emph{specializing function}, 
which is a function $f : T \to \omega$ such that $x <_T y$ implies $f(x) \ne f(y)$. 
Every special $\omega_1$-tree is Aronszajn, and any Suslin tree is a nonspecial Aronszajn tree. 
A function $f : T \to U$ between trees is 
\emph{strictly increasing} if $x <_T y$ implies $f(x) <_U f(y)$, and 
is \emph{height preserving} if $\h_T(x) = \h_U(f(x))$ for all $x \in T$. 
For trees $T$ and $U$ of height $\omega_1$, we say that $T$ and $U$ are 
\emph{club isomorphic} if there exists a club $C \subseteq \omega_1$ 
such that $T \res C$ and $U \res C$ are isomorphic. 

The next result is an essential tool for forcings involving Aronszajn trees.

\begin{thm}[Baumgartner {\cite[Chapter 4]{BDISS}}]
	Suppose that $T$ is an Aronszajn tree and $\{ x_\alpha : \alpha < \omega_1 \}$ 
	is a collection of pairwise disjoint finite subsets of $T$. 
	Then there exist $\alpha < \beta$ such that 
	every element of $x_\alpha$ is incomparable in $T$ with every element of $x_\beta$.
\end{thm}

In this article we are interested in preserving the Suslin property 
of a given Suslin tree after forcing. 
When iterating forcing, we only need to verify Suslin tree preservation at successor stages.

\begin{thm}[Abraham-Shelah {\cite[Theorem 3.1]{AS2}}, Miyamoto {\cite[Lemma 1.2]{MIYAMOTO}}]
	Let $S$ be a Suslin tree. 
	Then the property of a forcing poset being proper and forcing that 
	$S$ is Suslin is preserved by any countable support forcing iteration.
\end{thm}

We also want to preserve a Suslin tree after a finite support forcing 
iteration of c.c.c.\ forcings.

\begin{thm}
	Let $S$ be a Suslin tree. 
	Then the property of a forcing poset being c.c.c.\ and forcing that $S$ is Suslin 
	is preserved by any finite support forcing iteration.
\end{thm}

This theorem was known previously, but 
since we do not have a reference for a proof we provide a brief sketch.

\begin{proof}[Proof (Sketch)]
The result is immediate for iterations of length a successor ordinal, so 
let $\delta$ be a limit ordinal and suppose that we have a finite support forcing iteration 
$\langle \p_i, \dot \q_j : i \le \delta, \ j < \delta \rangle$ 
of c.c.c.\ forcings such that for all $i < \delta$, $\p_i$ forces that $S$ is Suslin. 
Then $\p_\delta$ is c.c.c. 
Let $p \in \p_\delta$ and assume for a contradiction that 
$$
p \Vdash_{\delta} \text{``$\{ \dot x_\alpha : \alpha < \omega_1 \}$ 
is an uncountable antichain of $S$.''}
$$
Then for each $\alpha$ we can choose $p_\alpha \le p$ in $\p_\delta$ and 
$y_\alpha \in S$ such that 
$p_\alpha \Vdash_{\p_\delta} \text{``$\dot x_\alpha = y_\alpha$.''}$ 
By a standard $\Delta$-system argument on the domains of the $p_\alpha$'s, 
find $\beta < \delta$ 
and an uncountable set $X \subseteq \omega_1$ such that for all $i < j$ in $X$, 
if $p_i \res \beta$ and $p_j \res \beta$ are compatible in $\p_\beta$, then 
$p_i$ and $p_j$ are compatible in $\p_\delta$. 

For all $i \in X$, $p_i \res \beta \le p \res \beta$ in $\p_\beta$. 
By a standard fact about c.c.c.\ forcings, there exists $u \le p \res \beta$ 
in $\p_\beta$ such that 
$$
u \Vdash_\beta 
\text{``$\{ i \in X : p_i \res \beta \in \dot G_{\p_\beta} \}$ is uncountable.''}
$$
Let $G$ be a generic filter on $\p_\beta$ which contains $u$. 
Let $Y := \{ i \in X : p_i \res \beta \in G \}$, which is uncountable. 
Then for all $i < j$ in $Y$, $p_i \res \beta$ and $p_j \res \beta$ are in $G$ and hence are 
compatible in $\p_\beta$. 
As $Y \subseteq X$, $p_i$ and $p_j$ are compatible in $\p_\delta$, which in turn 
easily implies that $y_i$ and $y_j$ are incomparable in $S$. 
Hence, in $V[G]$ the set $\{ y_i : i \in Y \}$ is an uncountable antichain of $S$, 
which contradicts our assumption that $S$ is Suslin in $V^{\p_\beta}$.
\end{proof}

We sometimes consider a normal tree $S$ as a forcing poset with the reversed order, 
which we also write as $S$. 
Then elements $a$ and $b$ of $S$ are compatible in this forcing poset iff they 
are comparable in the tree $S$. 
So an antichain of the tree $S$ is the same as an antichain of the forcing poset $S$. 
Hence, the tree $S$ is Suslin iff the forcing poset $S$ is c.c.c.
If $S$ is Suslin, then 
forcing with $S$ yields a generic filter which is a cofinal branch of $S$. 

Suppose that $S$ is a Suslin tree. 
Then for any dense open set $D \subseteq S$ (in the forcing poset), 
there exists some $\alpha < \omega_1$ such that $S_\alpha \subseteq D$. 
Namely, pick a maximal antichain $A \subseteq D$. 
Since $A$ is countable, we can fix $\alpha < \omega_1$ such that 
$A \subseteq S \res \alpha$. 
As $A$ is maximal and $D$ is open, 
it easily follows that $S_\alpha \subseteq D$.

Given finitely many $\omega_1$-trees $T_0,\ldots,T_{n-1}$, the product 
$T_0 \times \cdots \times T_{n-1}$ ordered componentwise by $<_T$ is a strict partial order. 
The suborder $T_0 \otimes \cdots \otimes T_{n-1}$ 
consists of all $n$-tuples in the product whose elements 
all have the same height. 
Since this suborder is dense in the product assuming the trees are normal, 
the suborder is c.c.c.\ iff the product is c.c.c. 
In particular, for $n > 1$, the tree $T_0 \otimes \cdots \otimes T_{n-1}$ is Suslin iff 
the tree $T_0 \otimes \cdots \otimes T_{n-2}$ is Suslin and 
$\Vdash_{T_0 \otimes \cdots \otimes T_{n-2}} \text{``$T_{n-1}$ is Suslin.''}$ 
This follows from the basic fact about c.c.c.\ forcings that $\p \times \q$ is c.c.c.\ 
iff $\p$ is c.c.c.\ and $\Vdash_\p \text{``$\q$ is c.c.c..''}$ 
Note that the height of a tuple in the tree $T_0 \otimes \cdots \otimes T_{n-1}$ is 
equal to the height of the elements of that tuple in the trees $T_0,\ldots,T_{n-1}$. 
It follows that $T_0 \otimes \cdots \otimes T_{n-2} \otimes T_{n-1}$ is isomorphic to 
$(T_0 \otimes \cdots \otimes T_{n-2}) \otimes T_{n-1}$.

Let $T$ be an $\omega_1$-tree. 
For every $a \in T$, define $T_a$ as the subtree consisting of all $b \in T$ 
such that either $b \le_T a$ or $a \le_T b$. 
For any positive integer $n$ and $n$-tuple 
$\vec a = (a_0,\ldots,a_{n-1})$ of elements of $T$ of the same height, define 
$T_{\vec a} := T_{a_0} \otimes \cdots \otimes T_{a_{n-1}}$, which is called a 
\emph{derived tree of $T$ of dimension $n$}. 
The tree $T$ is said to be \emph{$n$-free} if all of its derived trees of dimension $n$ are Suslin, 
and is \emph{free} if it is $n$-free for all positive integers $n$. 
Jensen \cite{DEVLINJ} proved that $\Diamond$ implies the existence of a free tree. 

In the remainder of this section we prove some easy facts about 
products of trees which will be helpful to refer to later.

\begin{lemma}
	Let $T$ be an $\omega_1$-tree. 
	Then for any derived tree $U = T_{a_0} \otimes \cdots \otimes T_{a_{n-1}}$ of $T$, 
	$U$ is Aronszajn iff for some $i < n$, $T_{a_i}$ is Aronszajn.
\end{lemma}

\begin{proof}
	Suppose that $U$ is not Aronszajn and let $B$ be a cofinal branch of $U$. 
	For each $\alpha < \omega_1$, 
	write $B(\alpha) = (B_0(\alpha),\cdots,B_{n-1}(\alpha))$. 
	Then for each $i < n$, $B_i$ is a cofinal branch of $T_{a_i}$, so 
	$T_{a_i}$ is not Aronszajn.
	Conversely, suppose that for each $i < n$, $B_i$ is a cofinal branch of $T_{a_i}$. 
	Then $B$ defined by $B(\alpha) := (B_0(\alpha),\ldots,B_{n-1}(\alpha))$ is a cofinal 
	branch of $U$.		
\end{proof}

\begin{lemma}
	Suppose that $S$ and $T$ are Suslin trees and $S \otimes T$ is special. 
	Then $\Vdash_{S} \text{``$T$ is special.''}$
\end{lemma}

\begin{proof}
	Let $f : S \otimes T \to \omega$ be a specializing function. 
	In $V^S$, let $B$ be a cofinal branch of $S$, and we define a 
	specializing function $g : T \to \omega$. 
	For any $x \in T$, define $g(x) := f(B(\h_T(x)),x)$. 
	If $x <_T y$, then $(B(\h_T(x)),x) <_{S \otimes T} (B(\h_T(y)),y)$, 
	so $g(x) = f(B(\h_T(x)),x)) \ne f(B(\h_T(y)),y) = g(y)$.
\end{proof}

\begin{lemma}
	Let $S$ be an $\omega_1$-tree and $b_0,\ldots,b_{n-1}$ distinct elements of $T$ 
	of the same height $\alpha$. 
	Suppose that $\alpha < \xi < \omega_1$. 
	If for all $c_0,\ldots,c_{n-1}$ above $b_0,\ldots,b_{n-1}$ respectively 
	of heights $\xi$, $S_{c_0} \otimes \cdots \otimes S_{c_{n-1}}$ is Suslin, 
	then $S_{b_0} \otimes \cdots \otimes S_{b_{n-1}}$ is Suslin.
\end{lemma}

\begin{proof}
	Suppose that $\{ (d_{i,0},\ldots,d_{i,n-1}) : i < \omega_1 \}$ 
	is an uncountable antichain of 
	$S_{b_0} \otimes \cdots \otimes S_{b_{n-1}}$. 
	Since level $\xi$ of $S$ is countable, we can find an uncountable 
	set $X \subseteq \omega_1$ and some $c_0,\ldots,c_{n-1}$ 
	of height $\xi$ such that for all $i \in X$, 
	$d_{i,0} \res \xi = c_0, \ldots, d_{i,n-1} \res \xi = c_{n-1}$. 
	Then $\{ (d_{i,0},\ldots,d_{i,n-1}) : i \in X \}$ is an uncountable 
	antichain of $S_{c_0} \otimes \cdots \otimes S_{c_{n-1}}$.
\end{proof}

\begin{lemma}
	Let $T$ be an $\omega_1$-tree and $U = T_{a_0} \otimes \cdots \otimes T_{a_{n-1}}$ 
	a derived tree of $T$. 
	Suppose that $m \le n$, $i_0 < \cdots < i_{m-1} \le n-1$, and 
	$W := T_{a_{i_0}} \otimes \cdots \otimes T_{a_{i_{m-1}}}$ is special. 
	Then $U$ is special.
\end{lemma}

\begin{proof}
	Let $f : W \to \omega$ be a specializing function. 
	Define $g : U \to \omega$ by letting 
	$g(c_0,\ldots,c_{n-1}) := f(c_{i_0},\ldots,c_{i_{m-1}})$. 
	Then easily $g$ is a specializing function for $U$, so $U$ is special. 
\end{proof}

\section{Preserving a Suslin Tree}

In this section we discuss the topic of forcing a property of some Aronszajn 
tree while preserving another tree being Suslin. 
Specifically, we consider forcings to make an Aronszajn tree special or to make 
two Aronszajn trees club isomorphic. 
This is not always possible; if there exists a strictly increasing 
function from a Suslin tree $S$ into an Aronszajn tree $T$, then specializing $T$ 
also specializes $S$.

The relation between a Suslin tree $S$ and an Aronszajn tree $T$ which we 
are interested in is whether adding a cofinal branch to $S$ also adds a cofinal 
branch to $T$. 
We use the following characterization of this relation.

\begin{proposition}[Lindstr\"{o}m \cite{LINDSTROM}]
	Let $S$ be a normal Suslin tree and $T$ a normal Aronszajn tree. 
	Then $\Vdash_S \text{``$T$ has a cofinal branch''}$ iff 
	there exists a club $C \subseteq \omega_1$ and a strictly increasing and height 
	preserving function $f : S \res C \to T \res C$.
\end{proposition}

In particular, for all $x \in S$, 
$x \Vdash_S \text{``$T$ has a cofinal branch''}$ iff 
there exists a club $C \subseteq \omega_1$ and a strictly increasing and height 
preserving function $f : S_x \res C \to T \res C$. 
It is easy to check that in this case the range $f[S_x \res C]$ has no uncountable 
antichain, and hence $T$ contains a Suslin subtree.

In Section 4.1 of \cite{AS2}, Abraham-Shelah 
proved that if $S$ is a Suslin tree and $T$ is an Aronszajn tree, and 
$S$ forces that $T$ is Aronszajn, then there is a forcing poset $\p$ which specializes 
$T$ while preserving $S$ being Suslin. 
The forcing $\p$ consists of countably infinite conditions and does not 
add new countable sets of ordinals. 
It is natural to ask whether the same property is true 
for Baumgartner's c.c.c.\ forcing for 
making $T$ special using finite conditions.

\begin{definition}[Baumgartner \cite{BDISS}]
	Let $T$ be a tree of height $\omega_1$. 
	Define $\q(T)$ to be the forcing poset whose conditions 
	are finite functions $p : \dom(p) \subseteq T \to \omega$ such that 
	$x <_T y$ in $\dom(p)$ implies $p(x) \ne p(y)$, ordered by reverse inclusion.
\end{definition}

\begin{thm}
	Let $T$ be a tree of height $\omega_1$. 
	Then $T$ has no cofinal branch iff $\q(T)$ is c.c.c.
\end{thm}

\begin{proof}
	See Chapter 4 of \cite{BDISS} for the forward direction. 
	Conversely, suppose that $B$ is a cofinal branch of $T$. 
	For each $\alpha < \omega_1$ define $p_\alpha := \{ (B(\alpha),0) \}$. 
	Then $\{ p_\alpha : \alpha < \omega_1 \}$ is an uncountable 
	antichain of $\q(T)$.
\end{proof}

So assuming that $T$ is an Aronszajn tree, $\q(T)$ is c.c.c.\ and forces 
that $T$ is special.

\begin{thm}
	Let $S$ be a Suslin tree and $T$ an Aronszajn tree. 
	Then 
	$$
	\Vdash_{S} \text{``$T$ is Aronszajn''} \ \Longleftrightarrow \ 
	\Vdash_{\q(T)} \text{``$S$ is Suslin.''}
	$$
\end{thm}

\begin{proof}
	We use the fact that for c.c.c.\ forcings $\p$ and $\q$, 
	$\p \times \q$ is c.c.c.\ iff $\Vdash_\p \text{``$\q$ is c.c.c.''}$ 
	Recall that the tree $S$ is Suslin iff the forcing poset 
	$S$ (with the reversed order) is c.c.c. 
	So both forcings $S$ and $\q(T)$ are c.c.c. 
	Therefore, 
	\begin{gather*}
	\Vdash_{\q(T)} \text{``$S$ is Suslin''} \ \Longleftrightarrow \ 
	\Vdash_{\q(T)} \text{``$S$ is c.c.c.''} \ \Longleftrightarrow \\
	\q(T) \times S \ \text{is c.c.c.} \ \Longleftrightarrow \ 
	S \times \q(T) \ \text{is c.c.c.} \ \Longleftrightarrow \ 
	\Vdash_S \text{``$\q(T)$ is c.c.c.''}
	\end{gather*}
	Now since $\q(T)$ is defined by finite conditions, by absoluteness 
	$\q(T)^V = \q(T)^{V^S}$. 
	Thus, $\Vdash_S \text{``$\q(T)$ is c.c.c.''}$ iff 
	$\Vdash_S \text{``$\q(T)^{V^S}$ is c.c.c.''}$, which by Theorem 3.3 
	is equivalent to $\Vdash_S \text{``$T$ is Aronszajn.''}$
\end{proof}	

Now we move on to the topic of making two normal Aronszajn trees club isomorphic 
while preserving some Suslin tree. 
We begin by reviewing the definition of a forcing poset $\q(T,U)$ for making $T$ and $U$ 
club isomorphic. 
This forcing is due to Abraham-Shelah; their definition is slightly different 
but their poset is isomorphic to a dense subset of $\q(T,U)$. 
See Section 5 of \cite{AS} for their definition and the proof of Theorem 3.6 below.

\begin{definition}
	Let $T$ and $U$ be normal Aronszajn trees. 
	Define the forcing poset $\q(T,U)$ to consist of all pairs $(x,f)$, where 
	$x$ is a finite set of countable limit ordinals, 
	$f$ is an injective function whose domain 
	is a finite downwards closed subset of $T \res x$ mapping into $U$, 
	and $f$ is strictly increasing and height preserving. 
	The ordering of $\q(T,U)$ is defined by $(y,g) \le (x,f)$ if $x \subseteq y$ and 
	$f \subseteq g$.
\end{definition}

\begin{thm}
	For any normal Aronszajn trees $T$ and $U$, 
	the forcing poset $\q(T,U)$ is proper.
\end{thm}

In Theorem 3.10 below 
we prove that if $S$ is a normal Suslin tree, $T$ and $U$ are normal Aronszajn trees, and 
forcing with $S$ does not add an uncountable branch to either $T$ or $U$, then 
forcing with $\q(T,U)$ preserves $S$. 
The proof relies on an analysis about compatibility of conditions 
in $\q(T,U)$.

\begin{lemma}
	Let $T$ and $U$ be normal Aronszajn trees and $(x,f)$ and $(y,g)$ conditions in $\q(T,U)$. 
	Suppose that $\alpha < \beta < \omega_1$ are limit ordinals 
	and the following statements hold:
	\begin{enumerate}
		\item $\alpha \in x$ and $\beta \in y$;
		\item $x \subseteq \beta$ and $x \cap \alpha = y \cap \beta$;
		\item $f \res (T \res \alpha) = g \res (T \res \beta)$;
		\item for all $a, b \in \dom(g) \cap T_\beta$, the ordinals 
		$\Delta_T(a,b)$ and $\Delta_U(g(a),g(b))$ are less than $\alpha$;
		\item every member of $\dom(f) \setminus (T \res \alpha)$ is incomparable 
		in $T$ with every member of $\dom(g) \setminus (T \res \beta)$, and 
		every member of $\ran(f) \setminus (U \res \alpha)$ is incomparable 
		in $U$ with every member of $\ran(g) \setminus (U \res \beta)$.
	\end{enumerate}
	Then $(x,f)$ and $(y,g)$ are compatible in $\q(T,U)$.	
\end{lemma}

\begin{proof}
	Using (5), it is easy to check that the pair 	
	$(x \cup y,f \cup g)$ satisfies all of the requirements of being a condition 
	except that the domain of the function 
	$f \cup g$ is not necessarily downwards closed in $T \res (x \cup y)$. 
	So we extend $f \cup g$ to a function $h$ whose domain is downwards closed 
	in $T \res (x \cup y)$ and then verify that $(x \cup y,h)$ is a condition. 
	Such an extension $h$ is obtained by adding to the domain of $f \cup g$ all elements 
	of $T$ of the form $a \res \gamma$, where $a \in \dom(g) \cap T_\beta$ and 
	$\gamma \in x \setminus \alpha$, 
	and defining $h(a \res \gamma) := g(a) \res \gamma$. 
	Note that by (4), for any new element $c$ of height $\gamma$ 
	there exists a \emph{unique} element 
	$a \in \dom(g) \cap T_\beta$ such that $c = a \res \gamma$.

	Obviously $h$ is height preserving and its domain 
	is downwards closed in $T \res (x \cup y)$. 
	If $h$ is not injective, then there are distinct $c$ and $d$ of the same 
	height $\gamma \in x \setminus \alpha$, at least one of which is new, 
	such that $h(c) = h(d)$. 
	Suppose that $c$ and $d$ are both new. 
	Then there are distinct $a$ and $b$ in $\dom(g) \cap T_\beta$ such that 
	$c = a \res \gamma$ and $d = b \res \gamma$. 
	But then $h(c) = h(d) \le_U g(a),g(b)$, which contradicts that 
	$\Delta_U(g(a),g(b)) < \alpha$ by (4). 
	If just one of them is new, then we may 
	assume $c = a \res \gamma$ where $a \in \dom(g) \cap T_\beta$ 
	and $d \in \dom(f)$. 
	Then $f(d) = h(d) = h(c) = g(a) \res \gamma <_U g(a)$, contradicting (5).

	It remains to prove that for all $c, d \in \dom(h)$, 
	$c <_T d$ implies $h(c) <_U h(d)$. 
	It suffices to verify this in the case where at least one of $c$ or $d$ is new.
	
	Case 1: $\h_T(c) < \alpha$ and $d = a \res \gamma$ is new. 
	Note $c \in \dom(g)$. 
	If $c <_T d$ then $c <_T a$, hence $g(c) <_U g(a)$. 
	Therefore, $h(c) = g(c) <_U g(a) \res \gamma = h(a \res \gamma) = h(d)$.
	
	Case 2: $c \in \dom(f) \setminus (T \res \alpha)$ and $d = a \res \gamma$ is new. 
	Then $c$ cannot be below $d$ in $T$, for otherwise $c <_T a$ 
	which contradicts (5).

	Case 3: $c = a \res \gamma$ is new and $d \in \dom(f) \setminus \alpha$. 
	If $c <_T d$, then $a \res \alpha = c \res \alpha <_T d$, and hence 
	$a \res \alpha = d \res \alpha$.  
	But $\dom(f)$ is downwards closed in $T \res x$ and $\alpha \in x$, 
	so $d \res \alpha$ is in $\dom(f)$. 
	then $d \res \alpha = a \res \alpha <_T a$, 
	which contradicts (5).

	Case 4: $c = a \res \gamma$ and $d = b \res \xi$ are both new. 
	Assume $c <_T d$. 
	Then $c <_T a, b$. 
	If $a \ne b$, then $\Delta_T(a,b) < \alpha \le \gamma = \h_T(c)$, which 
	contradicts that $c <_T a,b$.  
	So $a = b$. 
	Therefore, $h(c) = g(a) \res \gamma <_U g(a) \res \xi = h(d)$.
	
	Case 5: $c = a \res \gamma$ is new and $d \in \dom(g) \setminus \beta$. 
	Assume $c <_T d$. 
	Then $a \le_T d$, for otherwise by (4) 
	$\Delta_T(a,d) = \Delta_T(a,d \res \beta) < \alpha \le \h_T(c)$, which 
	contradicts that $c <_T a,d$. 
	Hence, $h(c) = g(a) \res \gamma <_U g(a) \le_U g(d) = h(d)$.
\end{proof}

\begin{lemma}
	Suppose that $T$ and $U$ are normal Aronszajn trees. 
	Let $Y \subseteq \omega_1$ be a stationary set of limit ordinals. 
	Assume that $\{ (x_\alpha,f_\alpha) : \alpha \in Y \}$ is a set of conditions 
	in $\q(T,U)$ such that for all $\alpha \in Y$, $\alpha \in x_\alpha$. 
	Then there exists $\alpha < \beta$ in $Y$ such that 
	$(x_\alpha,f_\alpha)$ and $(x_\beta,f_\beta)$ are compatible.
\end{lemma}

\begin{proof}
	By a straightforward pressing down argument, we can find a stationary set 
	$Y_0 \subseteq Y$, a function $f$, a set $x$, 
	and an ordinal $\gamma < \omega_1$ less than $\min(Y_0)$ 
	such that for all $\alpha \in Y_0$, 
	\begin{enumerate}
		\item $x_\alpha \cap \alpha = x$;
		\item $f_\alpha \res (T \res \alpha) = f$;
		\item for all distinct $a$ and $b$ in $\dom(f_\alpha) \cap T_\alpha$, 
		both of the ordinals $\Delta_T(a,b)$ and 
		$\Delta_U(f_\alpha(a),f_\alpha(b))$ are less than $\gamma$;
		\item for all $\beta \in Y_0$ larger than $\alpha$, 
		$x_\alpha \subseteq \beta$.
	\end{enumerate}
	Now applying Theorem 2.1 to the disjoint union of the trees $T$ and $U$, 
	we can find $\alpha < \beta$ in $Y_0$ 
	such that every member of $\dom(f_\alpha) \setminus (T \res \alpha)$ is incomparable 
	in $T$ with every member of $\dom(f_\beta) \setminus (T \res \beta)$, and every 
	member of $\ran(f_\alpha) \setminus (U \res \alpha)$ is incomparable in $U$ 
	with every member of $\ran(f_\beta) \setminus (U \res \beta)$. 
	By Lemma 3.7, the conditions 
	$(x_\alpha,f_\alpha)$ and $(x_\beta,f_\beta)$ are compatible.
\end{proof}

We need one more general result about Suslin trees.

\begin{lemma}
	Let $S$ be a Suslin tree. 
	Consider $\{ b_\alpha : \alpha \in Z \} \subseteq S$, where 
	$Z \subseteq \omega_1$ is stationary. 
	Then there exists some $a \in S$ such that for all $d >_S a$, 
	the set $Z_d := \{ \alpha \in Z : d \le_S b_\alpha \}$ 
	is stationary.
\end{lemma}

\begin{proof}
	Suppose not. 
	Then for all $a \in S$ we can fix $d_a >_S a$ and a club $C_a \subseteq \omega_1$ 
	such that $C_a \cap Z_{d_a} = \emptyset$. 
	Now the set $\{ d_a : a \in S \}$ is obviously dense, so its upwards 
	closure is dense open. 
	Since $S$ is Suslin, 
	we can fix some $\gamma < \omega_1$ such that for all $y \in S_\gamma$, 
	there exists some $a \in S$ such that $d_a \le_S y$.
	
	Let $D := \bigcap \{ C_a : a \in S \res \gamma \}$, which is a club since 
	$S \res \gamma$ is countable.
	As $Z$ is stationary, $D \cap Z$ is stationary. 
	So we can fix some $\alpha \in D \cap Z$ 
	such that $b_\alpha$ has height greater than or equal to $\gamma$. 
	Let $y := b_\alpha \res \gamma$. 
	By the choice of $\gamma$, there exists some $a \in S$ such that 
	$d_a \le_S y$. 
	Then $a \in S \res \gamma$, so $D \subseteq C_a$, and therefore $\alpha \in C_a$. 
	But $d_a \le_S y \le_S b_\alpha$, which means that $\alpha \in Z_{d_a}$. 
	So $\alpha \in C_a \cap Z_{d_a} = \emptyset$, which is a contradiction.
\end{proof}

\begin{thm}
	Suppose that $S$ is a normal Suslin tree. 
	Let $T$ and $U$ be normal Aronszajn trees such that 
	$$
	\Vdash_{S} \text{``$T$ and $U$ are Aronszajn.''}
	$$ 
	Then $\Vdash_{\q(T,U)} \text{``$S$ is Suslin.''}$
\end{thm}

\begin{proof}
	We prove the contrapositive. 
	Assume that there is a condition $p \in \q(T,U)$ such that 
	$$
	p \Vdash_{\q(T,U)} \text{``$\dot A = \{ \dot a_\alpha : \alpha < \omega_1 \}$ 
	is an uncountable antichain of $S$.''}
	$$
	We will find some $a \in S$ which forces in $S$ that either $T$ or $U$ 
	is not Aronszajn.
	
	Write $p = (x,f)$ and let $Z$ be the set of limit ordinals in $\omega_1$ above $\max(x)$. 
	For each $\alpha \in Z$, let $p_\alpha = (x \cup \{ \alpha \},f)$, 
	which is clearly a condition below $p$. 
	Extend each $p_\alpha$ to some $q_\alpha = (x_\alpha,f_\alpha)$ 
	which forces, for some $b_\alpha \in S$, that $\dot a_\alpha$ equals $b_\alpha$. 
	This gives us a family of conditions 
	$\{ (x_\alpha,f_\alpha) : \alpha \in Z \}$ 
	satisfying that for all $\alpha \in Z$, $\alpha \in x_\alpha$.
	
	Consider any $\alpha < \beta$ in $Z$ and suppose that $q_\alpha$ and 
	$q_\beta$ are compatible in $\q(T,U)$. 
	Fix $r \le q_\alpha, q_\beta$. 
	Then $r$ forces that $b_\alpha$ and $b_\beta$ are both in the antichain $\dot A$ 
	and hence are incomparable in $S$. 
	But then $b_\alpha$ and $b_\beta$ really are incomparable in $S$.

	Applying Lemma 3.9 to the collection $\{ b_\alpha : \alpha \in Z \}$, 
	we can find some $a \in S$ such that for all $d >_S a$, 
	the set 
	$$
	Z_d := \{ \alpha \in Z : d \le_S b_\alpha \}
	$$
	is stationary. 
	We claim that 
	$$
	a \Vdash_S 
	\text{``$\{ \alpha \in Z : b_\alpha \in \dot G_S \}$ is stationary.''}
	$$
	If not, then there exists some $d >_S a$ and an $S$-name $\dot D$ for a club 
	subset of $\omega_1$ such that 
	$$
	d \Vdash_S \text{``for all $\alpha \in \dot D \cap Z$, $b_\alpha \notin \dot G_S$.''}
	$$
	Since $S$ is c.c.c., we can find a club $E \subseteq \omega_1$ such that 
	$d$ forces that $E \subseteq \dot D$. 
	Now $Z_d$ is stationary, so we can fix some $\alpha \in Z_d \cap E$. 
	Then $d$ forces that $\alpha \in \dot D$ and hence that $b_\alpha \notin \dot G_S$. 
	On the other hand, $\alpha \in Z_d$ so $d \le_S b_\alpha$. 
	But then $b_\alpha$ extends $d$ in the forcing $S$, so $b_\alpha$ forces 
	that $b_\alpha \notin \dot G_S$, which is impossible. 
	This completes the proof of the claim.

	Let $G$ be a generic filter on $S$ such that $a \in G$. 
	We will prove that in $V[G]$, either $T$ or $U$ is not Aronszajn. 
	Suppose for a contradiction that both $T$ and $U$ are Aronszajn in $V[G]$. 
	Note that the definition of $\q(T,U)$ is absolute between $V$ and $V[G]$ due to 
	the finiteness of the conditions. 
	That is, $\q(T,U)^V = \q(T,U)^{V[G]}$. 
	In $V[G]$, define $Y := \{ \alpha \in Z : b_\alpha \in G \}$. 
	By the claim, $Y$ is stationary. 
	Moreover, the collection 
	$\{ (x_\alpha,f_\alpha) : \alpha \in Y \}$ 
	is a subset of $\q(T,U)^{V[G]}$ 
	which satisfies that for all $\alpha \in Y$, $\alpha \in x_\alpha$.

	Since $T$ and $U$ are normal Aronszajn trees in $V[G]$, we can 
	apply Lemma 3.8 in $V[G]$ to find some $\alpha < \beta$ in $Y$ such that 
	$q_\alpha = (x_\alpha,f_\alpha)$ and $q_\beta = (x_\beta,f_\beta)$ 
	are compatible in $\q(T,U)^{V[G]}$. 
	By absoluteness, $q_\alpha$ and $q_\beta$ are compatible in $\q(T,U)$ in $V$. 
	As observed above, the compatibility of $q_\alpha$ and $q_\beta$ in $\q(T,U)$  
	implies that $b_\alpha$ and $b_\beta$ are incomparable in $S$. 
	But $b_\alpha$ and $b_\beta$ are both in $G$, so they are comparable in $S$. 
	This contradiction completes the proof that either $T$ or $U$ is not 
	an Aronszajn tree in $V[G]$.
\end{proof}

\section{Consistency Results}

In this section we will apply the theorems of the previous section to 
prove some consistency results concerning Suslin trees. 
In our first result, we construct a model in which there exists a Suslin tree 
and any two normal Aronszajn trees, neither of which contains a Suslin subtree, 
are club isomorphic. 
In the second result, we prove that if $S$ is a free Suslin tree, then for any 
positive integer $n$ there exists a c.c.c.\ forcing poset which forces that $S$ 
is $n$-free but any derived tree of dimension $n+1$ is special.

Previously, Abraham-Shelah proved that it is consistent that there exists 
a Suslin tree and any Aronszajn tree either contains a Suslin tree or is special 
(see Section 4 of \cite{AS}). 
We strengthen their result by constructing a model with a Suslin tree 
in which there exists an essentially unique Aronszajn tree with no Suslin subtree 
(which of course must be special).

\begin{thm}
	Suppose that $S$ is a normal Suslin tree, $2^{\omega} = \omega_1$, and 
	$2^{\omega_1} = \omega_2$. 
	Then there exists a forcing poset which forces:
	\begin{itemize}
		\item $S$ is a Suslin tree;
		\item if $T$ and $U$ are normal Aronszajn trees, neither of which contains 
		a Suslin subtree, then $T$ and $U$ are club isomorphic.
	\end{itemize}
\end{thm}

In contrast to the aforementioned model of Abraham-Shelah which satisfies \textsf{GCH}, 
in our model we have that $2^\omega = 2^{\omega_1}$. 
This is necessary since by Section 2 of \cite{AS}, the weak diamond principle 
$2^\omega < 2^{\omega_1}$ implies the existence of 
$2^{\omega_1}$ many pairwise non-club isomorphic special Aronszajn trees.

\begin{proof}
Define by recursion a countable support forcing iteration 
$$
\langle \p_\alpha, \dot \q_\beta : \alpha \le \omega_2, \ \beta < \omega_2 \rangle
$$
of proper forcings which preserve $S$. 
After defining $\p_\alpha$, we consider by some bookkeeping 
a pair of Aronszajn trees $T_\alpha$ and $U_\alpha$ in $V^{\p_\alpha}$ and ask whether 
or not $S$ forces over $V^{\p_\alpha}$ that $T_\alpha$ and $U_\alpha$ remain Aronszajn. 
If so, then by Theorem 3.10 forcing with $\q(T_\alpha,U_\alpha)$ over 
$V^{\p_\alpha}$ preserves $S$ being Suslin. 
In this case, define $\dot \q_\alpha$ as a $\p_\alpha$-name for $\q(T_\alpha,U_\alpha)$, 
and otherwise let $\dot \q_\alpha$ be a $\p_\alpha$-name 
for the trivial forcing.

Consider on the other hand the case that there exists 
$x \in S$ and $W \in \{ T_\alpha, U_\alpha \}$ such that 
in $V^{\p_\alpha}$, $x \Vdash_S \text{``$W$ has a cofinal branch.''}$ 
By Proposition 3.1, there exists a club $C \subseteq \omega_1$ 
and a strictly increasing and height preserving function 
$f : S_x \res C \to W \res C$ in $V^{\p_\alpha}$. 
By upwards absoluteness, $f$ has the same property in $V^{\p_{\omega_2}}$. 
Since $S$, and hence $S_x$, is Suslin in $V^{\p_{\omega_2}}$, 
it follows that in $V^{\p_{\omega_2}}$ the tree $W$ contains a Suslin subtree, namely 
$f[S_x \res C]$.

By standard proper forcing iteration theorems and our cardinal arithmetic assumptions, 
for all $\beta < \omega_2$, $\p_\beta$ is $\omega_2$-c.c.\ and has 
cardinality at most $\omega_2$ (see Chapter VIII of \cite{SHELAHBOOK}). 
Thus, by a standard bookkeeping argument we can arrange that all pairs of Aronszajn 
trees in the final model have been handled at some stage less than $\omega_2$.
\end{proof}

Recall that for a coherent normal Suslin tree $S$, the forcing axiom 
$\textsf{PFA}(S)$ of Todorcevic \cite{PFAS} states that 
for any proper forcing $\p$ which preserves $S$ being Suslin, 
for any collection $\mathcal D$ of $\omega_1$ many dense subsets of 
$\p$, there exists a filter $G$ on $\p$ which meets every dense set in $\mathcal D$.

\begin{thm}
	The forcing axiom $\textsf{PFA}(S)$ implies that there exists a Suslin tree 
	(namely, $S$) and any two normal Aronszajn trees, neither 
	of which contains a Suslin subtree, are club isomorphic.
\end{thm}

\begin{proof}
	Assume $\textsf{PFA}(S)$ and consider normal Aronszajn trees 
	$T$ and $U$ which have no Suslin subtree. 
	By Proposition 3.1 and the comments which follow it, 
	$S$ forces that $T$ and $U$ are Aronszajn. 
	By Theorem 3.10, $\q(T,U)$ preserves $S$ being Suslin. 
	Also $\q(T,U)$ is proper. 
	By choosing a filter meeting an appropriate collection of dense subsets of $\q(T,U)$, 
	it is easy to show that $T$ and $U$ are club isomorphic.
\end{proof}

Note that the consistency result of Theorem 4.1 does not use large cardinals, in 
contrast to $\textsf{PFA}(S)$. 
Also, we are not using the coherence of $S$ and the same conclusion holds for 
$\textsf{PFA}(U)$ for any Suslin tree $U$ whether it is coherent or not.

We move on to our second application which concerns the topic of free trees. 
Free trees were originally introduced by Jensen \cite{DEVLINJ} 
as a counterpoint to homogeneous Suslin trees, 
which are trees such that for any distinct $a$ and $b$ of the same height, there exists  
an automorphism of the tree which maps $a$ to $b$ and $b$ to $a$. 
Note that the property of being homogeneous is upwards absolute. 
In contrast, as the next theorem shows, free trees are highly malleable by forcing.

\begin{thm}
	Suppose that $S$ is a free normal Suslin tree. 
	Then for any positive integer $n$, there exists a c.c.c.\ forcing poset which 
	forces that $S$ is $n$-free but all derived trees of $S$ of dimension greater than 
	$n$ are special.
\end{thm}

\begin{proof}
	Fix a positive integer $n$. 
	We define by recursion a finite support forcing iteration 
	$$
	\langle \p_\alpha, \dot \q_\beta : \alpha \le \omega_1, \ \beta < \omega_1 \rangle
	$$ 
	of c.c.c.\ forcings. 
	We will arrange that for each $\beta < \omega_1$, there exists an 
	$n+1$-tuple 
	$\vec a_\beta = (a_{\beta,0},\ldots,a_{\beta,n})$ of distinct elements of $S$ 
	of the same height such that  
	$$
	\Vdash_{\beta} \text{``$S_{\vec a_\beta}$ is Suslin and $\dot \q_\beta = 
	\q(S_{\vec a_\beta})$.''}
	$$
	By Theorem 3.3, each $\dot \q_\beta$ is forced to be c.c.c. 
	We bookkeep our forcings in such a way that for all 
	$\gamma < \beta < \omega_1$, 
	the height of the elements of $\vec a_\gamma$ 
	are less than or equal to the height of the elements of $\vec a_\beta$, 
	which is possible since the levels of $S$ are countable.

	Let us say that an $n+1$-tuple $\vec a$ of distinct elements of 
	$S$ of the same height has been \emph{handled} by stage $\delta$ 
	if $\Vdash_\delta \text{``$S_{\vec a}$ is special.''}$ 
	A given $n+1$-tuple $\vec a$ can be handled either explicitly by forcing with 
	$\q(S_{\vec a})$, or incidentally as a consequence of forcing other trees 
	to be special. 
	Our bookkeeping ensures that 
	all $n+1$-tuples of elements of one level of $S$ are handled before we move on 
	and handle the $n+1$-tuples of the next level. 
	We need to maintain at each step that every derived tree of $S$ 
	of dimension $n$ remains Suslin. 
	In order to make sure we can handle all $n+1$-tuples, we also 
	maintain that any derived tree of dimension $n+1$ which has not 
	been handled by a given stage $\delta$ is still Suslin in $V^{\p_\delta}$. 
	The following inductive hypothesis achieves these goals.
	
	\bigskip
	
	\noindent \textbf{Inductive Hypothesis on $\delta < \omega_1$}: 
	Let $\vec b = (b_0,\ldots,b_{m-1})$ be a tuple of distinct elements of $S$ 
	satisfying that for all $\gamma < \delta$:
	\begin{enumerate}
	\item the height of the elements of $\vec b$ 
	are greater than or equal to the height of the elements of $\vec a_\gamma$;
	\item for all $\gamma < \delta$ there exists $i \le n$ such that 
	for all $j < m$, $a_{\gamma,i} \not \le_S b_j$.
	\end{enumerate} 
	Then $\Vdash_\delta \text{``$S_{\vec b}$ is Suslin.''}$

	\bigskip
	
	Assume for now that the inductive hypothesis is true for all $\delta < \omega_1$, 
	and we describe how it can be used to prove the theorem. 
	To begin, we show that we can arrange all derived trees of 
	dimension $n+1$ to be special. 
	So consider $\delta < \omega_1$ and assume that a particular 
	$n+1$-tuple $\vec a = (a_0,\ldots,a_{n})$ has not been handled by stage $\delta$. 
	Since we are specializing derived trees one level of $S$ at a time, it follows that 
	for all $\gamma < \delta$, the height of the elements of $\vec a_\gamma$ is 
	less than or equal to the height of the elements of $\vec a$.

	We claim that $\Vdash_\delta \text{``$S_{\vec a}$ is Suslin.''}$ 
	By the inductive hypothesis, it suffices to show that 
	for all $\gamma < \delta$ there exists $i \le n$ such that 
	for all $j \le n$, $a_{\gamma,i} \not \le_S a_j$. 
	Suppose for a contradiction that $\gamma < \delta$ and for all $i \le n$ 
	there is some $j_i \le n$ such that $a_{\gamma,i} \le_S a_{j_i}$. 
	Since the elements of $\vec a_\gamma$ are distinct and $S$ is a tree, 
	it follows that the map $i \mapsto j_i$ from $n+1$ to $n+1$ 
	is an injection, and hence a bijection. 
	Consequently, $S_{a_{j_0}} \otimes \cdots \otimes S_{a_{j_n}}$ 
	is a subtree of $S_{\vec a_\gamma}$. 
	But $S_{\vec a_\gamma}$ is special in $V^{\p_{\gamma+1}}$, and hence in $V^{\p_\delta}$. 
	So $S_{a_{j_0}} \otimes \cdots \otimes S_{a_{j_n}}$ is also special in $V^{\p_\delta}$, 
	since any subtree of a special tree is special. 
	As $S_{\vec a}$ and $S_{a_{j_0}} \otimes \cdots \otimes S_{a_{j_n}}$ are isomorphic, 
	$S_{\vec a}$ is special in $V^{\p_\delta}$ as well, 
	which contradicts our assumption that $\vec a$ has not been handled by stage $\delta$. 

	In summary, any derived tree of dimension $n+1$ which has not been handled 
	by stage $\delta < \omega_1$ is still Suslin, and hence Aronszajn, in $V^{\p_\delta}$. 
	Thus, we can easily arrange by bookkeeping that the forcing iteration $\p_{\omega_1}$ 
	eventually handles all derived trees of $S$ of dimension $n+1$. 
	Therefore, $\p_{\omega_1}$ forces that all derived trees of $S$ of dimension 
	$n+1$ are special. 
	By Lemma 2.7, it follows that $\p_{\omega_1}$ forces that all derived trees of 
	$S$ of dimension greater than $n$ are special.
	
	Next let us see that 
	the inductive hypothesis implies that all derived 
	trees of $S$ of dimension $n$ are Suslin in $V^{\p_{\omega_1}}$. 
	So let $\vec b = (b_0,\ldots,b_{n-1})$ be an $n$-tuple of distinct elements 
	of $S$ of the same height. 
	To show that $S_{\vec b}$ is Suslin in $V^{\p_{\omega_1}}$, it suffices 
	to show that for all $\delta < \omega_1$, $S_{\vec b}$ is Suslin in $V^{\p_\delta}$. 
	As $\delta$ is countable, we can fix some $\xi < \omega_1$ 
	greater than the height of the elements 
	of $\vec b$ such that for all 
	$\gamma < \delta$, $\xi$ is greater than the height of the elements of $\vec a_\gamma$.
 
	By Lemma 2.6, in order to prove that $S_{\vec b}$ is Suslin in $V^{\p_\delta}$, 
	it suffices to show that for all $\vec c$ 
	in $S_{\vec b}$ whose elements have height $\xi$, 
	$S_{\vec c}$ is Suslin in $V^{\p_\delta}$. 
	So let such $\vec c = (c_0,\ldots,c_{n-1})$ be given. 
	By the inductive hypothesis, 
	it suffices to show that for all $\gamma < \delta$ there exists some $i \le n$ 
	such that for all $j < n$, $a_{\gamma,i} \not \le_S c_j$. 
	Let $\gamma < \delta$. 
	Let $\alpha$ be the height of the elements of $\vec a_\gamma$, 
	which is less than $\xi$ by the choice of $\xi$. 
	The set $\{ c_0 \res \alpha, \ldots, c_{n-1} \res \alpha \}$ has size at most $n$, 
	whereas $\{ a_{\gamma,0}, \ldots, a_{\gamma,n} \}$ has size $n+1$. 
	So we can choose $i \le n$ such that $a_{\gamma,i}$ is not 
	an element of $\{ c_0 \res \alpha, \ldots, c_{n-1} \res \alpha \}$. 
	Then for all $j < n$, $a_{\gamma,i} \not \le_S c_j$ and we are done.

	It remains to prove the inductive hypothesis.  
	Let $\delta < \omega_1$ and assume that the inductive hypothesis holds 
	for all $\beta < \delta$. 
	Let $\vec b = (b_0,\ldots,b_{m-1})$ be a tuple of distinct elements of $S$ of  
	the same height satisfying that for all $\gamma < \delta$:
	(a) the height of the elements of $\vec b$ 
	are greater than or equal to the height of the elements of $\vec a_\gamma$, and 
	(b) there exists $i \le n$ such that 
	for all $j < m$, $a_{\gamma,i} \not \le_S b_j$. 
	We will prove that $\Vdash_\delta \text{``$S_{\vec b}$ is Suslin.''}$

	Note that for all $\delta_0 < \delta$, $\vec b$ satisfies properties (a) and (b) 
	for all $\gamma < \delta_0$. 
	By the inductive hypothesis, for all $\delta_0 < \delta$, 
	$\Vdash_{\delta_0} \text{``$S_{\vec b}$ is Suslin.''}$ 
	If $\delta$ is a limit ordinal, then by Theorem 2.3 it follows that 
	$\Vdash_{\delta} \text{``$S_{\vec b}$ is Suslin''}$ and we are done.

	Suppose that $\delta = \beta+1$ is a successor ordinal. 
	Then as just observed, $\Vdash_{\beta} \text{``$S_{\vec b}$ is Suslin.''}$ 
	So it suffices to prove that in $V^{\p_\beta}$, 
	$\Vdash_{\q_\beta} \text{``$S_{\vec b}$ is Suslin.''}$ 
	Recall that $\q_\beta$ is equal to $\q(S_{\vec a_\beta})$. 
	By our assumptions on $\vec b$, we can 
	fix $i^* \le n$ such that for all $j < m$, 
	$a_{\beta,i^*} \not \le_S b_j$. 
	Let $\alpha$ be the height of the elements of $\vec b$.

	We work in $V^{\p_\beta}$. 
	In order to show that 
	$\Vdash_{\q(S_{\vec a_\beta})} \text{``$S_{\vec b}$ is Suslin''}$, 
	by Theorem 3.4 it suffices to show that 
	$\Vdash_{S_{\vec b}} \text{``$S_{\vec a_\beta}$ is Aronszajn.''}$ 
	By Lemma 2.4, it is enough to show that 
	$\Vdash_{S_{\vec b}} \text{``$S_{a_{\beta,i^*}}$ is Aronszajn.''}$ 
	By a simple argument, it suffices to show that whenever $a \in S_\alpha$ 
	is above $a_{\beta,i^*}$, then 
	$\Vdash_{S_{\vec b}} \text{``$S_a$ is Aronszajn.''}$ 
	So let such an $a$ be given. 
	By the choice of $i^*$, $a$ is not equal to any of $b_0,\ldots,b_{m-1}$.

	There are two possibilities to consider. 
	First, assume that 	for all $\gamma < \beta$, there exists some $i \le n$ 
	such that $a_{\gamma,i}$ is not less than or equal to 
	any of $b_0,\ldots,b_{m-1},a$. 
	Applying the inductive hypothesis for $\beta$ we get that 
	$S_{b_0} \otimes \cdots \otimes S_{b_{m-1}} \otimes S_a$ 
	is Suslin in $V^{\p_\beta}$. 
	As discussed in Section 2, it follows that 
	$S_{\vec b}$ forces that $S_a$ is Suslin, and hence 
	Aronszajn, and we are done.

	Secondly, assume that there exists some $\gamma < \beta$ such that 
	for all $i \le n$, $a_{\gamma,i}$ is less than or equal to one of $b_0,\ldots,b_{m-1},a$. 
	By assumption (b) about $\vec b$, there exists $i \le n$ such that 
	for all $j < m$, $a_{\gamma,i} \not \le_S b_j$, and therefore $a_{\gamma,i} \le_S a$. 
	For any $l \le n$ different from $i$, $a_{\gamma,i}$ and $a_{\gamma,l}$ are different, 
	so they cannot both be below $a$. 
	Hence, we can pick $j_l < m$ such that $a_{\gamma,l} \le_S b_{j_l}$. 
	Then the map $l \mapsto j_l$ is injective. 
	Define $d_i := a$ and for $l \le n$ different from $i$, $d_l := b_{j_l}$. 
	Then $S_{d_0} \otimes \cdots \otimes S_{d_n}$ is a subtree of $S_{\vec a_\gamma}$. 
	Now in $V^{\p_{\gamma+1}}$, and hence in $V^{\p_\beta}$, $S_{\vec a_\gamma}$ 
	is special. 
	Since $S_{d_0} \otimes \cdots \otimes S_{d_n}$ is a subtree of 
	$S_{\vec a_\gamma}$, it is special as well. 
	By Lemma 2.7, it follows that $S_{b_0} \otimes \cdots \otimes S_{b_{m-1}} \otimes S_a$ 
	is special. 
	By Lemma 2.5, $S_{\vec b}$ forces that $S_a$ is special and hence Aronszajn.
\end{proof}	

We mention a related result of Scharfenberger-Fabien \cite{gido} that 
under $\Diamond$, for each positive integer $n$ 
there exists an $n$-free tree which is not $n+1$-free.


\end{document}